\documentclass[11pt]{amsart}
\usepackage{etex}
\usepackage{diagbox}
\usepackage{array}
\usepackage{graphicx, pictex}
\usepackage{stmaryrd}
\usepackage{changepage}
\usepackage{float}
\restylefloat{table}
\usepackage{multirow}
\usepackage[dvipsnames]{xcolor}
\usepackage{amsmath,amssymb,amsthm,mathrsfs}
\allowdisplaybreaks

\def\sumt#1#2#3#4{\sum_{#1=#2}^#3 #4}

\newcommand\pmatbig[2]{\begin{pmatrix} #1 \\ #2 \end{pmatrix}}

\newcommand\pmat[2]{\left( \begin{smallmatrix} #1\\#2 \end{smallmatrix} \right)}

\def\bp{{\bf p}}

\def\bPhi{{\boldsymbol{\Phi}}}

\def\bq{{\mathbf q}}
\def\bp{{\mathbf p}}

\def\bM{{\mathbf M}}

\def\d{\Omega}

\def\du#1#2#3{\overset{#3}{\underset{#2}{#1}}}
\def\Forall{\quad \hbox{ for all }}

\def\M{{\mathcal{M}}}
\def\T{\mathcal{T}}

\newcommand{\norm}[1]{\lVert{#1}\rVert}
\newcommand{\tn}[1]{\lVert\kern-1pt\lvert{#1}\rvert\kern-1pt\rVert}
\def\<{{\langle}}
\def\>{{\rangle}}
\def\Forall{\quad \hbox{ for all }}

\def\bq{{\mathbf q}}

\def\bM{{\mathbf M}}

\def\bq{{\mathbf q}}

\def\bM{{\mathbf M}}

\def\d{\Omega}

\def\Forall{\quad \hbox{ for all }}

\def\bq{\bar q}

\def\d{\Omega}

\def\Forall{\quad \hbox{ for all }}

\def\bq{\bar q}

\def\tb#1{{\|\kern-1pt| #1 \|\kern-1pt|}}
\def\nm2#1#2{\|#1\|_{2,\d_{#2}}}

\def\R{\mathbb{R}}

\def\T{\mathcal{T}}

 \theoremstyle{plain}
 \newtheorem{thm}{Theorem}[section]
 \numberwithin{equation}{section} 
 \numberwithin{figure}{section} 
 \theoremstyle{plain}
 \newtheorem{prop}[thm]{Proposition}
 \theoremstyle{plain}
 \newtheorem{algorithm}[thm]{Algorithm} 
 \theoremstyle{plain}
 \newtheorem{theorem}[thm]{Theorem}
 \theoremstyle{plain}
 
\theoremstyle{plain}
 \newtheorem{remark}[thm]{Remark}
 \theoremstyle{plain}
 \newtheorem{lemma}[thm]{Lemma}

\def\T{{\mathcal{T}}}
\def\M{{\mathcal{M}}}

\def\d{{\Omega}}

\def\Forall{\quad \hbox{ for all }}

\def\<{{\langle}}
\def\>{{\rangle}}
\def\R{\mathbb{R}}

\def\bq{{\mathbf q}}

\def\bM{{\mathbf M}}

\def\du#1#2#3{\overset{#3}{\underset{#2}{#1}}}

\def\bM{{\mathbf M}}

\usepackage{amssymb}

\begin{document}

\title[Efficient solvers for Reaction-Diffusion  singular problem]
{Efficient  discretization and preconditioning of the singularly perturbed Reaction-Diffusion problem}

\author{Constantin Bacuta}
\address{University of Delaware,
Mathematical Sciences,
501 Ewing Hall, Newark, DE 19716}
\email{bacuta@udel.edu}

\author{Daniel Hayes}
\address{University of Delaware,
Department of Mathematics,
501 Ewing Hall 19716}
\email{dphayes@udel.edu}

\author{Jacob Jacavage}
\address{Lafayette College,
Department of Mathematics,
Pardee Hall, Easton, PA 18042 }
\email{jacavagj@lafayette.edu}

\keywords{least squares, saddle point systems, mixed methods,  optimal test norm, Uzawa conjugate gradient, preconditioning}

\subjclass[2000]{74S05, 74B05, 65N22, 65N55}
\thanks{The work was supported  by NSF-DMS 2011615}%

\begin{abstract}

We consider the reaction diffusion problem and present efficient ways to discretize and precondition in the singular perturbed case when the reaction term dominates the equation. Using  the concepts of optimal test norm and saddle point reformulation, we  provide efficient discretization processes for uniform and non-uniform meshes. We present a preconditioning  strategy  that works for a  large range of the perturbation parameter. Numerical examples to illustrate  the efficiency of the method are included for a problem on the unit square. 

\end{abstract}
\maketitle

\section{Introduction}

We consider the singularly perturbed reaction diffusion problem 
\begin{equation}\label{PDE_RD}
   \left\{
\begin{array}{rcl}
     -\varepsilon\, \Delta u +cu=\ f & \mbox{in} \ \ \ \Omega,\\
      u =\ 0 & \mbox{on} \ \partial\Omega,\\ 
\end{array} 
\right. 
\end{equation}
for non-negative constant $\varepsilon$ and $c_{max} \geq c(x) \geq c_{min} >0$ on $\Omega$, a bounded domain in $\R^d$. We focus on the reaction dominated case, i.e., $\varepsilon \ll 1$. The discretization of the equation arises in solving  practical PDE models, such as heat transfer problems in thin domains \cite{MR1716824} as well as when using small step sizes in implicit time discretizations of parabolic reaction diffusion type problems \cite{Lin-Stynes12}. The solutions to these problems are characterized by exponential boundary layers 
which pose  numerical challenges due to the $\varepsilon$-dependence of the  error estimates  and of the stability constants. 

The standard variational formulation for \eqref{PDE_RD} is: Find $u\in H_0^1(\Omega)$ such that
\begin{equation}\label{RDstandard}
\varepsilon (\nabla u,\nabla v)+(cu,v)=(f,v)\Forall v\in H_0^1(\Omega).
\end{equation}
To discretize \eqref{RDstandard}, we propose a mixed variational formulation  that allows for higher order of approximation  and  for efficient preconditioning. Towards efficient  preconditioning, we introduce a simplified version of the Bramble, Pasciak, and Vasilevski (BVP)  multilevel precondtioner \cite{BPViski2000}, that works for nested spaces corresponding to uniform refinement meshes. 

A standard discretization of \eqref{RDstandard} can be solved by using a PCG algorithm on $V_h\subset H_0^1(\Omega)$ with the BVP or the simplified  preconditioner we introduce. The process is efficient for a certain range of $\varepsilon \geq \varepsilon_0$. To improve the rate of convergence (in the energy norm) for  the case $\varepsilon <\varepsilon_0$, we proceed by using two main tools. The first is a mixed formulation that is suitable to Saddle Point Least Squares (SPLS) discretization \cite{BJprec, BJ-AA,BJ-nc, BQ15}. This allows for a higher order approximation of the flux $\nabla u$ using the same discretization spaces as in the  standard discretization of  \eqref{PDE_RD}. The second tool is based on using refined meshes towards the boundary layer(s) of the problem. Even though the refinement approach presented in this paper is focused on the case in which $\Omega$ is the unit square in $\R^2$, as well as the use of Shiskin meshes, the ideas of mixed reformulation and multilevel preconditioning extend to more  general cases of $\Omega \subset \R^d$.  

Classical first-order system least squares methods for the reaction diffusion problem can be found in \cite{AdMaLMad-2015, cai-ku2020, Lin-Stynes12}. However, our approach of saddle point least squares reformulation is different. We reformulate the standard variational formulation \eqref{RDstandard} as a mixed variational formulation first, and then further reformulate it as a saddle point problem that allows preconditioning and extra approximability of the flux by using standard finite element spaces. The approach is related  to the Lagrange multiplier approach that leads to Stokes type systems. In our case, the Lagrange multiplier  is the variable of interest. This idea was used before in other particular problems, see e.g., \cite{BM12,Dahmen-Welper-Cohen12,demkowicz, Barrett-Morton81}. 

Many aspects of the SPLS formulation are also common to the DPG approach \cite{ bouma-gopalakrisnan-harb2014,car-dem-gop16,demkowicz-gopalakrishnanDPG10, demkowicz2011class,J5onDPG, dem-fuh-heu-tia19}. The SPLS method as a general method for solving mixed variational formulations is presented in \cite{BJprec, BJ-AA,BJ-nc, BQ15}. In this paper, we also combine SPLS formulation with the concept of optimal test norm \cite{Broersen-StevensonDPGcd14, Chan-Heuer-Bui-Demkowicz14, Dahmen-Welper-Cohen12, demkowicz-gopalakrishnanDPG10,J5onDPG, Barrett-Morton81, dem-fuh-heu-tia19} and a general preconditioning technique, introduced in \cite{BJprec}, in order to improve the stability and approximability of the final discretization process. 

The paper is organized as follows. In Section \ref{sec:BPV}, we review the Bramble-Pasciak-Vasilevski multilevel preconditioning technique and introduce a simplified preconditioner for the reaction diffusion equation. In Section \ref{sec:ReviewLSPP}, we introduce the notation and review the SPLS formulation, discretization, preconditioning, and the concept of optimal test norms. Sections \ref{SPLS4RD} and \ref{SPLSd4RD} detail how to apply the general SPLS theory to the reaction diffusion problem. Numerical results are included in Section \ref{num4rd}.

\section{Spectral representation of  $H^1$ norm and Multilevel Preconditioning}\label{sec:BPV}
\subsection{Abstract formulation} 
Assume $V$ is a Hilbert space with inner product $a(\cdot,\cdot)$ and that we have nested spaces $V_0\subset V_1 \subset \dots \subset V_J = V_h \subset V$. We let $(\cdot,\cdot)$ be another inner product on $V$ that induces a weaker (than $a(\cdot,\cdot)$) norm on $V$, and denote by   $Q_j: V \to V_j$ the orthogonal projection with respect to the $(\cdot,\cdot)$ inner product for $j=0,1,\cdots, J$. We will  we further assume that the following  norm equivalence
\begin{equation}\label{eq;NormRep}
\sumt{k}{0}{J}{\lambda_k\norm{(Q_k - Q_{k-1})v}^2} \simeq a(v,v), 
\end{equation}
holds on the entire space $V_h=V_J$. Here, the $\lambda_j$'s  are positive constants, $Q_{-1} =0$, and $Q_J$ coincides with the  identity on $V_J$. Furthermore, we assume the equivalence constants associated with the symbol `` $\simeq $" are independent of $h$ or $J$.

Let $c(\cdot,\cdot)$ be another symmetric bilinear form defined on $V$ and define the operators $A_h:V_h\to V_h$ and $C_h:V_h\to V_h$ by
$$
(A_hu,v) = a(u,v), \quad (C_hu,v) = c(u,v) \Forall v\in V_h.
$$
We assume that there are positive constants $\mu_1,\dots,\mu_J$ such that
\begin{equation}\label{eq;NormRepC}
\sumt{k}{0}{J}{\mu_k\norm{(Q_k - Q_{k-1})v}^2} \simeq c(v,v), 
\end{equation}
with equivalent constants independent of $h$ or $J$. The goal of this section is to present a simple multilevel preconditioning  technique for the following problem: Find $u_h\in V_h$ such that
$$
\tilde{a}(u_h,v):= \epsilon a(u_h,v) + c(u_h,v) = (f,v) \Forall v\in V_h,
$$
or in operator form 
\[
(\epsilon A_h + C_h)u_h = f_h:= Q_h f. 
\]

\subsection{Bramble-Pasciak-Vassilevski multilevel preconditioner}
Using that 
\[
\begin{aligned}
 Q_kQ_j  &= Q_jQ_k = Q_j , j\leq k\\
(Q_k-Q_{k-1})(Q_j - Q_{j-1})&=(Q_j - Q_{j-1}) \delta_{j,k}, \\
\end{aligned}
\]
the  equivalences \eqref{eq;NormRep} and \eqref{eq;NormRepC} can be written in the operator forms
\[
\begin{aligned}
L_h: & =\sumt{j}{0}{J}{\lambda_j(Q_j - Q_{j-1})}\simeq A_J = A_h,\\
K_h : &=\sumt{j}{0}{J}{\mu_j(Q_j - Q_{j-1})}\simeq C_J = C_h,
\end{aligned}
\]
respectively. Using the above (projection type) spectral representations and the properties of the projections, we obtain that
\[
\begin{aligned}
\epsilon L_h + C_h & = \sumt{j}{0}{J}{(\epsilon\lambda_j + \mu_j)(Q_j - Q_{j-1})}\simeq \epsilon A_h + C_h, \\
P_h : &= \sumt{j}{0}{J}{(\epsilon\lambda_j + \mu_j)^{-1}(Q_j - Q_{j-1})} = (\epsilon L_h + C_h)^{-1},
\end{aligned}
\]
are symmetric positive definite (discrete) operators on $V_h$. In addition, $P_h$ is a (uniform)  preconditioner for $\epsilon A_h + C_h$. 

The BPV approach further modifies $P_h$  by using another (sum of local projections type) operator that avoids computing  the action of $Q_k$ (hence mass matrix inversion). In the BPV approach, the projections $Q_k - Q_{k-1}$ are replaced by $(\tilde{Q}_k - \tilde{Q}_{k-1})^2$, where $\tilde{Q}_k$ is given by
\begin{equation}\label{eq:Qtilda}
\tilde{Q}_k\,v_h := \sumt{i}{1}{{n_k}}{\frac{(v_h,\varphi_i^k)}{(1,\varphi_i^k)}\varphi_i^k} \Forall  v_h \in V_h,
\end{equation}
where $\{\varphi_i^k, i=1,2, \cdots, n_k\}$ is  a basis for $V_k$.

\subsection{A simplified BVP (sBVP)  preconditioner}\label{sBVP}
A more general form of  the preconditioner $P_h$ is given by
\[
P_h =  \sumt{j}{0}{J}{\gamma_j(Q_j - Q_{j-1})},
\]
where $\gamma_j$ is a positive real number for each $j$. We further note that by using ``summation by parts''  for the general form of $P_h$, we obtain
\[
\begin{aligned}
P_h\, f_h &= \sumt{k}{0}{J}{\gamma_k(Q_k-Q_{k-1})f_h}\\
& = \left(\sumt{k}{0}{J}{\gamma_kQ_k} - \sumt{k}{1}{J}{\gamma_kQ_{k-1}}\right)f_h\\
& = \left(\sumt{k}{0}{J}{\gamma_kQ_k} - \sumt{k}{0}{{J-1}}{\gamma_{k+1}Q_k}\right)f_h\\
& = \gamma_J\, f_h +\sumt{k}{0}{{J-1}}{(\gamma_{k}-\gamma_{k+1})Q_kf_h}.
\end{aligned}
\]

We are able to further  simplify the preconditioner $P_h$ under three assumptions that are easily satisfied in our applications.  We first assume that $\{\gamma_j\}$ is a decreasing sequence of positive numbers:
\[
(A1)  \ \ \ \gamma_0 \ge \gamma_1 \ge \cdots \ge  \gamma_J>0.
\]
Second, we let $\{\tilde{Q}_k: V_h \to V_k \}$, with $k=0, 1, \cdots, J-1$, be any family of linear operators satisfying 
\[
(A2) \ \  \   (\tilde{Q}_k v_h, w_h)  = (v_h, \tilde{Q}_k w_h)  \Forall v_h, w_h \in V_h,
\]
and 
\[
(A3) \ \  \ c_1  (Q_k v_h, v_h) \leq (\tilde{Q}_k v_h, v_h)  \leq c_2 (Q_k v_h, v_h) \Forall v_h \in V_h,
\]
where $c_1, c_2$ are positive constants independent of $h$. We define a new operator $\tilde{P}_h: V_h \to V_h$ by
\[
\tilde{P}_h\, f_h  = \lambda_J\, f_h +\sumt{k}{0}{{J-1}}{( \gamma_{k}- \gamma_{k+1})\tilde{Q}_kf_h}.
\]
\begin{lemma}
Under the assumptions $(A1), (A2) $ and $(A3)$ we have that 
\begin{equation} \label{eq:esimate1}
c_1 (P_h v_h, v_h)  \leq (\tilde{P}_h v_h, v_h)  \leq c_2 (P_h v_h, v_h) \Forall v_h \in V_h.
\end{equation}
Consequently, if $P_h$ is a uniform preconditioner for $A_h$, then $\tilde{P}_h$ is also a uniform preconditioner  for $A_h$. 
\end{lemma}
\begin{proof} First, we note that due to $(A1)$ and $(A2)$ we have that $\tilde{P}_h$ is a symmetric positive definite operator on $V_h$. 
Multiplying the inequalities in $(A3)$ by the appropriate positive scalars and summing up the new  inequalities, we obtain \eqref{eq:esimate1}.
\end{proof}


\subsection{Preconditioning the reaction diffusion problem}
The standard variational formulation of problem \eqref{PDE_RD} is: Find $u\in H_0^1(\Omega)$ such that
$$
\epsilon (\nabla u,\nabla v) + (cu,v) = (f,v) \Forall v\in H_0^1(\Omega),
$$
where $(\cdot, \cdot)$ denotes the standard $L^2$ inner product for scalar or vector functions. We consider the nested  sequence of spaces $\{V_k\}_{k=\overline{0,J}}$ of continuous piecewise linear functions associated with the uniformly refined meshes  $\{\T_k\}_{k=\overline{0,J}}$  on $\Omega$, and define the operators
$$
(A_h u,v) = (\nabla u, \nabla v)  \Forall u, v\in V_h,
$$
$$
(C_hu,v) = (cu,v) \Forall u, v\in V_h.
$$
That is, $A_h$ is the discrete Laplacian operator and $C_h$ is  the discretization of the multiplication by the function $c(x)$  operator. It is well known, see \cite{BPViski2000, brenner-scott}, that for $\lambda_j = 1/h_j^2 \approx 4^j$ we have 
\[
L_h :=  \sumt{i}{0}{J}{\lambda_j(Q_j - Q_{j-1})}  \simeq A_J= A_h.
\]
Using that $ \sumt{j}{0}{J}{(Q_j - Q_{j-1})} = I$, for $\mu_j := c^*$, where $c^*$ is any value in $ [c_{min},c_{max}]$,  we obtain 
$$
K_h :=\sumt{j}{0}{J}{\mu_j(Q_j - Q_{j-1})}  \simeq (c^*I)  \simeq C_h.
$$
Thus, the BPV preconditioner for  $\varepsilon A_h + C_h$  becomes
\begin{equation}\label{eq:BVP}
\begin{aligned}
B^{BVP}_J &:= \sumt{j}{0}{J}{(\epsilon/h_j^2 + c^*)^{-1}(Q_j - Q_{j-1})}\\ 
&= \gamma_J I  +\sumt{j}{0}{{J-1}}{(\gamma_{j}-\gamma_{j+1})Q_j},
\end{aligned}
\end{equation}
where $\gamma_j=(\epsilon/h_j^2 + c^*)^{-1}$.

It was proved  in \cite{BPViski2000} that  $B^{BVP}_J$  is a uniform preconditioner for $\epsilon A_h + C_h$.
Using the family $\{\tilde{Q}_k: V_h \to V_k \}_{k=\overline{0, J-1}}$ defined in \eqref{eq:Qtilda}, with 
$\{\varphi_i^k, i=1,2, \cdots, n_k\}$  the standard nodal basis on $V_k$, we can define the simplified BPV (sBVP) preconditioner for the reaction diffusion problem 
\begin{equation}\label{eq:sBVP}
B^{sBVP}_J : = \gamma_J I +\sumt{j}{0}{{J-1}}{( \gamma_{j}- \gamma_{j+1})\tilde{Q}_j}, 
\end{equation}
where  $\gamma_j=(\epsilon/h_j^2 + c^*)^{-1}$.
To prove that $B^{sBVP}_J$ is a uniform preconditioner for $\epsilon A_h + C_h$, we only need to check that $(A1)-(A3)$ are satisfied. Assumption $(A1)$ is sataisfied as the function $h\to (\epsilon/h^2 + c^*)^{-1}$ is decreasing on $(0, 1]$. From the definition of $\tilde{Q}_k$ in \eqref{eq:Qtilda}, one can easily verify $(A2)$. We will prove that $(A3)$ holds next.
\begin{lemma} Let  $\{V_k\}_{k=\overline{0,J}}$  be a nested sequence of spaces of continuous piecewise linear functions associated with a set of quasi-uniform meshes $\{\T_k\}_{k=\overline{0,J}}$ on $\Omega$, and assume that  $\tilde{Q}_k$ is defined as in \eqref{eq:Qtilda}. Then assumption $(A3)$ holds.
\end{lemma} 
\begin{proof}
Define  the diagonal  matrix $D$ with entries $D_{ii}=(1, \varphi_i^k)$, and let $M$ be the mass matrix for the basis $\{\varphi_i^k, i=\overline{1:n_k} \}$. Using that the mesh $\T_k$ is quasi-uniform (could be just locally), we obtain
\[
(v_h, v_h) \simeq  \sum_{i=1}^{n_k} (1, \varphi_i) \ v_h^2(z_i) \Forall v_h \in V_h, 
\]
where the uniformity constants are indepnedent of $k$. Here, $z_i$ corresponds to the nodal function $\varphi_i^k$, i.e., $\varphi_i^k(z_j) =\delta_{ij}$.  
This is equivalent to 
\begin{equation}\label{eq:MequivD}
M \simeq  D \  \text{or} \  D^{-1} \simeq M^{-1},
\end{equation}
which implies 
\[
MD^{-1} M \simeq M.
\]

From \eqref{eq:Qtilda}, we obtain
\[
(\tilde{Q}_k\,v_h, v_h) = \sumt{i}{1}{{n_k}}{\frac{(v_h,\varphi_i^k)^2}{(1,\varphi_i^k)}} =\sumt{i}{1}{{n_k}}{\frac{(Q_kv_h,\varphi_i^k)^2}{(1,\varphi_i^k)}}.
\]
Let $Q_kv_h= \sum_{j=1}^{n_k} \alpha_i^k \varphi_j^k$ and $\alpha = (\alpha_j^k)_{j=\overline{1:n_k}}$. Then, using $MD^{-1} M \simeq M$, 
\[
(\tilde{Q}_k\,v_h, v_h) = \sumt{i}{1}{{n_k}}{\frac{(\sum_{j=1}^{n_k} \alpha_j^k \varphi_i^k,\varphi_i^k)^2}{(1,\varphi_i^k)}} =(D^{-1} M\alpha, M\alpha)_e \simeq (M\alpha, \alpha)_e, 
\]
where $(\cdot, \cdot)_e$ is the Euclidian inner product. Since 
\[ 
(M\alpha, \alpha)_e= (Q_kv_h, Q_kv_h) = (Q_kv_h, v_h), 
\] 
assumption $(A3)$ holds. 
\end{proof}

Remarks on the implementation of $B^{sBVP}_J$ are included in the Appendix. Other related works on preconditioning  singularly perturbed reaction-diffusion problem can be found in \cite{ kolev-xu-zhu2016} for the finite element discretization on quasi-uniform meshes, where a more general Additive Schwartz preconditioner is proposed, and in \cite{MacMad-2013} for finite difference discretization on fitted meshes, where a block-structured preconditioning approach is proposed. 

\section{The notation and the general SPLS approach}\label{sec:ReviewLSPP}
We now review the main ideas and concepts for the SPLS formulation, discretization, and preconditioning of a general  mixed variational formulation. Let $V$ and $Q$ be Hilbert spaces and $F\in V^*$. We are interested in problems of the form: Find $p\in Q$ such that
\begin{equation}\label{cont_problem}
b(v,p)=\langle F,v\rangle \Forall v\in V,
\end{equation}
where $b(\cdot,\cdot)$ is a continuous bilinear form on $V\times Q$ satisfying the following  $\inf-\sup$ condition on $V\times {Q}$:
\begin{equation} \label{inf-sup_a}
\du{\inf}{p \in Q}{} \ \du {\sup} {v \in V}{} \ \frac {b(v, p)}{|v|\,\|p\|} =m>0.
\end{equation} 

We view  $Q$, the trial space in \eqref{cont_problem}, as a subspace of larger (host) space $\tilde{Q}$ and equip $Q$ with the induced  inner product and norm from $\tilde{Q}$. The extra space $\tilde{Q}$ is needed for the SPLS {\it non-conforming discretization}. We assume the inner products $a_0(\cdot, \cdot)$ and $(\cdot, \cdot)_{\tilde{Q}}$ induce the  norms $|\cdot|_V =|\cdot| =a_0(\cdot, \cdot)^{1/2}$ and $\|\cdot\|_{\tilde{Q}}=\|\cdot\|=(\cdot, \cdot)_{\tilde{Q}}^{1/2}$. We denote the dual of $V$ by $V^*$ and the dual pairing on $V^* \times V$ by $\langle \cdot, \cdot \rangle$.  
We  further assume that $b(\cdot, \cdot)$ has a continuous extension to a bilinear form on $V\times \tilde{Q}$ satisfying
\begin{equation}\label{sup-sup_a}
\du{\sup}{p \in \tilde{Q}}{} \ \du {\sup} {v \in V}{} \ \frac {b(v, p)}{|v|\,\|p\|} =M <\infty.
\end{equation}
With the form $b$, we associate the operator $B:V\to \tilde{Q}$ defined by  
\[
(B v,q)_{\tilde{Q}}=b(v, q)  \Forall  v \in V, q \in \tilde{Q}.
\]

In this paper, we assume that $V_0 :=Ker(B)= \{ 0 \} $. Most of the considerations in this section extend to a nontrivial kernel  $V_0$, see \cite{BJ-nc}. It is well known that if a bounded form $b:V \times {Q} \to \R$ satisfies \eqref{inf-sup_a}, then problem \eqref{cont_problem} has a unique solution, see e.g., \cite{A-B, B09}. The standard {\it saddle point reformulation} of \eqref{cont_problem} (see \cite{BM12, BQ15,BQ17, Dahmen-Welper-Cohen12}) is: Find $(w, p) \in (V, Q)$ such that  
\begin{equation}\label{abstract:variational2}
\begin{array}{lclll}
a_0(w,v) & + & b( v, p) &= \langle F,v \rangle &\ \Forall  v \in V,\\
b(w,q) & & & =0   &\  \Forall  q \in Q.  
\end{array}
\end{equation}
\subsection{The concept of optimal test norm} \label{sec:opt} If we assume  that $Range (B) \subset Q$ and that the operator $B:V\to Q$ is injective ($V_0=Ker(B)=\{0\}$) then, as in  \cite{Broersen-StevensonDPGcd14, Chan-Heuer-Bui-Demkowicz14, Dahmen-Welper-Cohen12, demkowicz-gopalakrishnanDPG10, Barrett-Morton81}, we can define an equivalent norm on $V$
\[
|v|_{opt}:= \du {\sup} {p \in Q}{} \ \frac {b(v, p)}{\|p\|} =  \du {\sup} {p \in Q}{} \ \frac {(B v, p)}{\|p\|}= \|Bv\|_Q= \|Bv\|_{\tilde{Q}},
\]
that is operator dependent. We will refer to this as the {\it optimal test norm}. By replacing the form $a_0(\cdot, \cdot)$ in \eqref{abstract:variational2}  with the inner product induced by the {\it optimal test norm}, i.e., $a_{opt}(u, v):= (Bu, Bv)_Q$, we obtain that  both the continuity constant $M$ and the $\inf-\sup$ constant $m$ are equal to $1$. Thus, the stability (at the continuous level) of the new saddle point formulation is optimal.


\subsection{The abstract variational  formulation at the discrete level}
The  non-conforming (trial space) {\it SPLS  discretization} of \eqref{cont_problem} is defined as a  saddle point discretization of \eqref{abstract:variational2} with $V_h \subset V$ and with $\M_h$ a subspace of $\tilde{Q}$, but in general {\it not necessarily a subspace of} $Q$. Assume that standard discrete $\sup-\sup$ and  $\inf-\sup$ conditions hold for the pair $(V_h, \M_h)$ with constants $M_h$ and $m_h$ respectively.
The discrete mixed variational formulation of \eqref{cont_problem} is: Find $ p_h \in \M_h$ such that 
\begin{equation}\label{abstract:variational1_h} 
b( v_h, p_h) = \langle F,v_h \rangle \ \Forall  v_h \in V_h.
\end{equation}
In general, this might not have a unique solution. However, discretization of \eqref{abstract:variational2}: Find $(w_h, p_h) \in V_h \times \M_h$ such that  
\begin{equation}\label{abstract:variational2_h}
\begin{array}{lclll}
a_0(w_h,v_h) & + & b( v_h, p_h) &= \langle F,v_h \rangle &\ \Forall  v_h \in V_h,\\
b(w_h,q_h) & & & =0   &\  \Forall  q_h \in\M_h, 
\end{array}
\end{equation}
always has a unique solution. The variational formulation \eqref{abstract:variational2_h} is the {\it non-conforming saddle point least squares} (n-c SPLS) {\it discretrization}  of \eqref{cont_problem}.

Using $(\cdot, \cdot)$, another (weaker) inner product on $V$, we can define the discrete operator $A_h : V_h \to V_h$ associated with the form $a_0(\cdot ,\cdot)$ on $V_h$ by
\[ 
(A_h v_h, w_h)=a_0(v_h, w_h) \   \Forall v_h, w_h \in V_h 
\]
and the linear operators $B_h:V_h\to \M_h$ and  $B_h^*: {\M}_h \to V_h$  by 
\[ 
(B_h v_h,q_h)_{\M_h}=b(v_h, q_h)= (B_h^*q, v_h) \   \Forall v_h \in V_h, \ q \in \M_h. 
\]
The Schur complement of  \eqref{abstract:variational2_h} is denoted by  $S_h=B_h\, A_h^{-1} B_h^*$. In what follows, $V_h\subset V$ will be chosen as a standard conforming finite element space. On the other hand, each choice of  the space $\M_h$, possibly non-conforming to $Q$, leads  to a new  SPLS discretization for which $p_h \in \M_h \subset  \tilde Q$.
\subsection{Constructing a discrete trial space from a general test space}\label{sec:Mh}
Let $V_h$ be a {\it finite element subspace} of $V$. Following \cite{BJ-AA, BJ-nc}, we provide a general construction of discrete trial spaces $\M_h$, defined using the operator $B$ associated with the original problem \eqref{cont_problem}. Let $\tilde{\M}_h \subset \tilde{Q}$ be a finite dimensional subspace equipped with the inner product $(\cdot,\cdot)_h$. 
The  corresponding induced norm on $\tilde{\M}_h$  will be denoted by $\|\cdot\|_h$. Define the representation operator $R_h:\tilde{Q}\to \tilde{\M}_h$ by 
\begin{equation*}\label{rep_operator}
(R_hp,q_h)_h:=(p,q_h)_{\tilde{Q}} \Forall q_h \in \tilde{\M}_h.
\end{equation*}
In the case when $(\cdot,\cdot)_h$ coincides with the inner product on $\tilde{Q}$, we have that $R_h$ is precisely the orthogonal projection onto $\tilde{\M}_h$. Since the space $\tilde{\M}_h$ is finite dimensional, there exist constants $k_1,k_2$ such that 
\begin{equation}\label{norm-equiv}
k_1\|q_h\|\leq \|q_h\|_h\leq k_2\|q_h\| \Forall q_h \in \tilde{\M}_h.
\end{equation}
Using the operator $R_h$, we define $\M_h$ as
\[
\M_h:=R_hBV_h\subset \tilde{\M}_h \subset \tilde{Q}.
\]
The following proposition, see \cite{B09} and \cite{BJ-nc}, gives a sufficient condition on $R_h$ to ensure that the family $\{(V_h, R_h  B V_h)\}$ is stable.
\begin{prop}\label{discrete-Stability}
Assume that 
\begin{equation}\label{eq:ProjCoerces}
\|R_h  q_h \|_h \geq \tilde{c}\,  \| q_h\| \ \Forall q_h \in  B V_h,
\end{equation}
with a constant $\tilde{c}$ independent of $h$. Then 
\begin{equation}
\inf_{p_h \in \M_h}\sup_{v_h\in V_h} \frac{b(v_h,p_h)}{|v_h|\  \|p_h\|_h} \geq \tilde{c}\, m  >0, 
\end{equation}
vhere $m$ is defined in \eqref{inf-sup_a}. 
\end{prop}
As a consequence of Proposition \ref{discrete-Stability}, we have that  \eqref{abstract:variational2_h}  has  a unique solution $(w_h, p_h)  \in (V_h, \M_h)$ and $w_h=0$. Regarding the {\it approximability property} of the projection type trial space, the following proposition was proved in \cite{BJ-nc}. 
\begin{prop}\label{approxProj}
If $p$ is the solution of \eqref{cont_problem} and $p_h$ is the second component of the solution of \eqref{abstract:variational2_h}, then
\[
\|p-p_h\| \leq C\inf_{q_h\in \M_h}\|p-q_h\|,
\]
whith $C= 1+ \frac{1}{ k_1\, \tilde{c}}$, where $k_1$ and $\tilde{c}$ were introduced  in  \eqref{norm-equiv} and \eqref{eq:ProjCoerces}, respectively.
\end{prop}
\begin{remark}
The choice of $\tilde{\M}_h$ is important. In practice $V_h$ is a space of continuous piecewise polynomials, and by applying the (differentiation) operator $B$ we obtain discontinuous functions. The representation operator $R_h:\tilde{Q}\to \tilde{\M}_h$  acting on $BV_h$  can be viewed also as a smoothing operator, as the range $\tilde{\M}_h$ is a subspace of $\tilde{Q}$ consisting  of continuous functions. The SPLS discretization with $\tilde{\M}_h$ as trial space can be viewed as a (smoothing type) recovery or post-processing process. 
\end{remark} 

\subsection{An Uzawa CG iterative solver}\label{iter_discrete_problem}
Note that a global linear system may be difficult to assemble when solving \eqref{abstract:variational2_h} on  $(V_h,\M_h=R_hBV_h)$, especially  if the operator $R_h$ involves a  global projection.  In this case, bases for the trial spaces $\M_h$ might be difficult to find. One can solve \eqref{abstract:variational2_h} and  avoid building a basis for $\M_h$ by using  an Uzawa Conjugate Gradient (UCG) algorithm.

\begin{algorithm} (UCG) Algorithm \label{alg:UCG}
\vspace{0.1in}

 {\bf Step 1:} {\bf Choose any}  $p_0 \in \M_h$. {\bf Compute} $w_{1} \in V_h $, $q_1, d_1 \in \M_h$ by
\[ 
 \begin{aligned}
& a_0( w_{1}, v_h)& = &\ \<f_h,v_h\> - b(v, p_{0}) &\ &\Forall v_h \in V_h,&\\
&   (q_1, q)_{h} & = & \ b(w_1 ,q)  &\ &\Forall  q \in \M_h,& \ \ d_1:=q_1.
\end{aligned}
\]
\vspace{0.1in}

{\bf Step 2:} {\bf For} $j=1,2,\ldots, $ {\bf compute} 
 $h_j, \alpha_j, p_j, w_{j+1}, q_{j+1}, \beta_j, d_{j+1}$ by 
\[ 
 \begin{aligned}
& {\bf (UCG1)} \ \ \ \  & a_0( h_{j}, v_h) = & - b(v_h, d_j)  && \text{for all} \ v_h \in V_h&\\
& {\bf (UCG\alpha)} \ \ \ \  & \alpha_j =& - \frac{(q_j, q_j)_{h}}{b(h_j,q_j)} \\
& {\bf (UCG2)} \ \ \ \  & p_{j} = & \ p_{j-1} + \alpha_j \  d_j \\
& {\bf (UCG3)} \ \ \ \  & w_{j+1} = & \ w_j + \alpha_j\ h_j \\
& {\bf (UCG4)} \ \ \ \  & (q_{j+1}, q)_{h} = & \ b(w_{j+1} ,q) && \text{for all} \ q \in \M_h& \\
& {\bf (UCG\beta)} \ \ \ \  & \beta_j=& \ \frac{(q_{j+1}, q_{j+1})_{h}}{(q_j,q_j)_{h}} \\
& {\bf (UCG6)} \ \ \ \  & d_{j+1}= & \ q_{j+1} +\beta_j d_j. \\
\end{aligned}
\]
\end{algorithm}
\begin{remark}\label{rem:inversionOnMh} 
From {\bf (UCG4)}, we have that $q_{j+1}= B_h w_{j+1}$. Thus, $q_{j+1}$  can be computed by inverting the Gram matrix corresponding to a basis of $\tilde{\M_h}$, which in our applications is component-wise a space of continuous piecewise linear functions. The Gram matrix corresponding to a basis of $\M_h=R_hBV_h$ is not needed for the computation of $q_{j+1}$ in {\bf (UCG4)} or $q_{1}$ in {\bf Step 1}. 
\end{remark}
The main inversion needed at each step involves $a_0(\cdot,\cdot)$ in \textbf{Step 1} or (\textbf{UCG1}). In operator form, these steps become 
\begin{equation}\label{inversions}
w_1=A_h^{-1}(f_h-B_h^*p_0),\qquad  \mathrm{and} \qquad h_j=-A_h^{-1}(B_h^*d_j).
\end{equation}
In order to build an efficient solver for \eqref{cont_problem}, we  modify Algorithm \ref{alg:UCG} by replacing the action of $A_h^{-1}$ with the action of a suitable preconditioner. 

\subsection{Preconditioning the SPLS discretization}\label{sec:precon}
We summarize a general preconditioning framework to approximate the solution of  \eqref{cont_problem}  that is  presented in \cite{BHJ, BJprec}. We plan to combine this framework with the new concept of {\it optimal} test norm. 
Let $P_h:V_h\to V_h$ be  a general  (preconditioing) operator that is equivalent to $A_h^{-1}$ in the sense that 
\begin{equation}\label{dual_relation}
( P_h f_h, g_h) =(f_h P_h g_h) \  \Forall f_h,g_h\in V_h,
\end{equation}
and
\begin{equation}\label{equiv_forms}
m_1^2|v_h|^2\leq a_0(P_hA_hv_h,v_h)\leq m_2^2|v_h|^2,
\end{equation}
where the  positive constants $m_1^2,m_2^2$ are the smallest and largest eigenvalues of $P_hA_h$, respectively. Condition \eqref{equiv_forms} is equivalent with the fact that the condition number of $P_hA_h$ satisfies
\begin{equation}\label{condition_bound}
\kappa(P_hA_h) = \frac{m_2^2}{m_1^2}.
\end{equation}

It was proved in \cite{BHJ, BJprec} that  the saddle point discretization  \eqref{abstract:variational2_h} would not lose stability and approximability  properties if $\kappa(P_hA_h)$ is independent of $h$  and $A_h^{-1}$ is replaced by $P_h$. The replacement of the action of $A_h^{-1}$   leads to solving
\begin{equation}\label{abstract:variationalPh}
\begin{aligned}
w_h & =P_h (f_h-B_h^*p_h),\\
B_h w_h  & =0. 
\end{aligned}
\end{equation}
The Schur complement associated with the modified problem \eqref{abstract:variationalPh} is  
\[
\tilde{S}_h=B_hP_hB_h^*.
\]
The corresponding  version of Algorithm \ref{alg:UCG} to solve \eqref{abstract:variationalPh} is the following Uzawa Preconditioned Conjugate Gradient (UPCG) algorithm.
\begin{algorithm} (UPCG) Algorithm for Mixed Methods\label{alg:PUCG}
\vspace{0.1in}

 {\bf Step 1:} {\bf Choose any} $p_0 \in \M_h$. {\bf Compute} $u_{1} \in V_h $, $q_1, d_1 \in \M_h$ by 
\[ 
 \begin{aligned}
& u_1 &=&P_h(F_h-B_h^*p_0) \\
& q_1& =& B_h u_1, \ \ d_1:=q_1.
\end{aligned}
\]
\vspace{0.1in}

{\bf Step 2:} {\bf For} $j=1,2,\ldots, $ {\bf compute} 
 $h_j, \alpha_j, p_j, u_{j+1}, q_{j+1}, \beta_j, d_{j+1}$ by 
\[ 
 \begin{aligned}
& {\bf (PCG1)} \ \ \ \  & h_j =& - P_h(B_h^*d_j)  && &\\ 
& {\bf (PCG\alpha)} \ \ \ \  & \alpha_j =& - \frac{(q_j, q_j)_Q}{b(h_j,q_j)} \\
& {\bf (PCG2)} \ \ \ \  & p_{j} = & \ p_{j-1} + \alpha_j \  d_j \\
& {\bf (PCG3)} \ \ \ \  & u_{j+1} = & \ u_j + \alpha_j\ h_j \\
& {\bf (PCG4)} \ \ \ \  & q_{j+1} = & B_h u_{j+1} ,& & & \\
& {\bf (PCG\beta)} \ \ \ \  & \beta_j=& \ \frac{(q_{j+1}, q_{j+1})_Q}{(q_j,q_j)_Q} \\
& {\bf (PCG6)} \ \ \ \  & d_{j+1}= & \ q_{j+1} +\beta_j d_j. \\
\end{aligned}
\]
\end{algorithm}
According to  \cite{BJprec}, the UPCG iterations $p_{j}$ converge to the solution $p_h$ of  \eqref{abstract:variationalPh}, and the rate of convergence for $\|p_{j} -p_h \|_{\tilde{S}_h}$ depends on the condition number of $\tilde{S}_h$ that satisfies
\begin{equation}\label{condition-number-estimate}
\kappa(\tilde{S}_h)\leq  \kappa(S_h)\cdot \kappa(P_hA_h). 
\end{equation}

\section{SPLS formulation for the reaction diffusion problems}\label{SPLS4RD}
In this section, we consider  the SPLS discretization for the reaction diffusion problem using an optimal test norm. The goal is to emphasize how stability and approximability for the mixed formulation can be gained by using the SPLS approach and provide a way of choosing appropriate preconditioners to find an efficient iterative solver. In what follows, $(\cdot, \cdot)$ and $\|\cdot\|$  will denote the standard $L^2$ inner product and norm, respectively. 

To place equation \eqref{PDE_RD} into the general SPLS framework, we define the spaces $V:=H_0^1(\Omega)$, $\tilde{Q}:=L^2(\Omega)\times L^2(\Omega)^d$, and $Q$ as the graph of the operator $\varepsilon\nabla:H_0^1(\Omega)\to L^2(\Omega)^d$, i.e., 
\[
Q:=G(\varepsilon\nabla)=\left\{\pmat{v}{\varepsilon\nabla v}\,|\,v\in H_0^1(\Omega)\right\}.
\]
Since the operator $\varepsilon\nabla$ is bounded from $H_0^1(\Omega)$ to $L^2(\Omega)^d$, the space $Q$ is closed by the Closed Graph Theorem. We define $b:V\times \tilde{Q}\to\mathbb{R}$ as
\[
b(v,\pmat{q}{\bq}):=(cq,v)+(\bq,\nabla v) \Forall v\in V,\pmat{q}{\bq}\in \tilde{Q},
\]
and the linear functional $F\in V^*$ as
\[
\langle F,v\rangle:=(f,v) \Forall v\in H_0^1(\Omega).
\]
With this setting, the SPLS formulation of \eqref{RDstandard} is: Find $\bp=\pmat{u}{\varepsilon\nabla u}\in Q$ such that  
\begin{equation*}
b(v,\bp)=(cu,v)+(\varepsilon \nabla u, \nabla v)=(f,v)\Forall v\in V.
\end{equation*}

On $\tilde{Q}$, we consider the weighted inner product
\begin{equation}\label{c-eps-1}
(\pmat{q}{\bq},\pmat{p}{\bp})_{\tilde{Q}}=(cq,p)+(\varepsilon^{-1}\bq,\bp):= (q,p)_c+(\bq,\bp)_{\varepsilon^{-1}},
\end{equation}
which gives us the corresponding norm
\[
\|\pmat{q}{\bq}\|_{\tilde{Q}}=\left(\|c^{1/2}q\|^2+\|\varepsilon^{-1/2}\bq\|^2\right)^{1/2}.
\]
The operator $B:V\to \tilde{Q}$ is then given by
\[
Bv=\pmat{v}{\varepsilon\nabla v} \Forall v\in V.
\]
We note that 
\[
V_0=\mathrm{Ker}(B)=\{v\in H_0^1(\Omega)\,|\,Bv=0\}=\{0\}.
\]
Thus, the {\it optimal  test norm} on $V$ is induced by the inner product 
\[
a_{opt} (u,v)=(Bu, Bv)_Q= (\varepsilon\nabla u,\nabla v)+(cu,v) \Forall u,v\in V,
\]
which gives rise to the norm
\[
|v|_{opt}=\left(\|c^{1/2}v\|^2+\|\varepsilon^{1/2}\nabla v\|^2\right)^{1/2}.
\]

In addition, according to Section \ref{sec:opt}, we have that the continuity constant $M$ of the bilinear form $b$ and the $\inf-\sup$ constant $m$ are equal to $1$. One can also directly check that 
\begin{equation}\label{inf_supRD}
\sup_{v\in V}\frac{b(v,\pmat{u}{\varepsilon\nabla u})}{|v|_{opt}}=\|\pmat{u}{\varepsilon\nabla u}\|_{{Q}},
\end{equation}
for any $\pmat{u}{\varepsilon\nabla u}\in Q$. While this does lead to optimal continuity and $\inf-\sup$ constants, inverting the operator associated with $|\cdot |_{opt}$ coincides with solving the original problem. Fortunately, at the discrete level we can replace the action of the discrete operator corresponding to $a_{opt} (\cdot,\cdot)$ by a preconditioner, as presented in Section \ref{sec:BPV}. 


\section{SPLS discretization for the Reaction Diffusion Problems}\label{SPLSd4RD}
In this section, we outline choices for the test and trial spaces involved in SPLS discretization. We take $V_h \subset V=H_0^1(\Omega)$ to be the space of continuous piecewise linear polynomials with respect to the mesh $\T_h$ vanishing on the boundary of $\Omega$. To construct a trial space $\M_h$, as discussed in Section \ref{sec:Mh}, we first define $\tilde{\M}_h\subset  \tilde{Q} =L^2(\Omega)\times L^2(\Omega)^d$ as
\[
\tilde{\M}_h:=M_{h}\times \varepsilon \bM_{h},
\]
where $M_{h}$ consists of continuous piecewise linear polynomials with respect to the mesh $\T_h$  (with no restrictions on the boundary) and $\bM_{h}$ is the vector-valued product space in which each component consists of continuous piecewise linear polynomials. Let $\{\phi_1,\dots,\phi_N\}$ denote a nodal basis for $M_{h}$ with respect to the mesh $\T_h$ and $\{\bPhi_1,...,\bPhi_{2N}\}$ denote a nodal basis for $\bM_{h}$, where $\bPhi_j=(\phi_j,0)^T$ and $\bPhi_{N+j}=(0,\phi_j)^T$ for $j=1,\dots, N$. Two different choices the projection type trial space, based on the inner product chosen for $\tilde{\M}_h$, are considered. More details about the construction of these trial spaces can be found in \cite{BJ20}.

For the first choice of projection trial space, we equip $\tilde{\M}_h$ with the inner product induced from $\tilde{Q}$. In this case, $R_h$ is the orthogonl projection from $\tilde{Q}$ onto $\tilde{\M}_h$, and we define the trial space as
\[
\M_h:=R^{\mathrm{\, orth}}_hBV_h,
\]
where we use the notation $R^{\mathrm{\, orth}}_h$ to signify equipping $\tilde{\M}_h$ with the induced inner product from $\tilde{Q}$.

For the second choice of projection trial space, we equip $\tilde{\M}_h$ with an inner product related to lumping the mass matrix. More precisely, for two elements $\pmat{q_h}{\bq_h}, \pmat{p_h}{\bp_h}\in \tilde{\M}_h$, where
\[
\bq_h=\sum_{i=1}^{2N}\alpha_i\varepsilon\bPhi_i, \ \text{and} \ \bp_h=\sum_{i=1}^{2N}\beta_i\varepsilon\bPhi_i,
\]
we define 
\[
\left(\pmat{q_h}{\bq_h},\pmat{p_h}{\bp_h}\right)_h:=(cq_h,p_h)+\sum_{i=1}^{2N}\alpha_i\beta_i(1,\varepsilon\bPhi_i).
\]
In this case, the action of $R_h:\tilde{Q}\to \tilde{\M}_h$ is given by
\[
R_h\pmatbig{q}{\bq}=\pmatbig{Q_hq}{Q_h^{\,\mathrm{lump}}\bq},
\]
where $Q_h:L^2(\Omega) \to M_{h}$ is the orthgonal projection with respect to the weighted inner product $(\cdot,\cdot)_c$ and 
\[
Q^{\,\mathrm{lump}}_h\bq = \sum_{i=1}^{2N}\frac{(\bq,\varepsilon\bPhi_i)_{\varepsilon^{-1}}}{(1,\varepsilon\bPhi_i)}\varepsilon\bPhi_i=\sum_{i=1}^{2N}\frac{(\bq,\bPhi_i)}{(1,\bPhi_i)}\bPhi_i.
\]
We then define the trial space as
\[
\M_h:=R^{\mathrm{\, lump}}_hBV_h,
\]
where we use the notation $R^{\mathrm{\, lump}}_h$ to signify equipping $\tilde{\M}_h$ with this type of inner product. For the remainder of this section, we will write $R_h$ without the superscript for simplicity as the results hold for either choice of inner product.

The discrete mixed variational formulation is: Find $\bp_h=R_hBu_h$, with $u_h\in V_h$, such that
\[
b(v_h,\bp_h)=(f,v_h) \Forall v_h\in V_h,
\]
and the corresponding SPLS discretization, with {\it optimal test norm}, is: Find $(w_h, \bp_h=R_hA\nabla u_h)\in V_h \times \M_h$ such that
\begin{equation}\label{discrete_problemSPLSRD}
\begin{array}{lclll}
a_{opt}(w_h,v_h) & + & b(v_h,\bp_h) &= (f,v_h) &\ \Forall  v_h \in V_h,\\
R_h B w_h & & & =\mathbf{0}.
\end{array}
\end{equation}
In the above problem, as mentioned we can replace the action of the discrete operator corresponding to $a_{opt} (\cdot,\cdot)$ by a preconditioner. The following result regarding the stability of these spaces was proved in \cite{BJ20}. 
\begin{theorem}\label{Stability2}
Let $\Omega \subset \mathbb{R}^2$ be a polygonal domain and $\{T_h\}$ be a family of locally quasi-uniform meshes for $\Omega$. For each $h$, let $V_h$ be the space of continuous linear functions with respect to the mesh $\{\T_h\}$ that vanish on $\partial \Omega$ and $\M_h=R_hBV_h$. Then the family of spaces $\{(V_h,\M_h)\}$ is stable.
\end{theorem}

\begin{remark}
While the above result assumes $\Omega \subset \mathbb{R}^2$ is a polygonal domain, the result can be extended to polyhedral domains in $\mathbb{R}^3$.
\end{remark}

\section{Numerical results}\label{num4rd} 
We considered equation \eqref{PDE_RD} on the unit square with variable coefficient $c=2(1+x^2+y^2)$ and $f$ computed such that the exact solution is given by
\[
\begin{aligned}
u(x,y)&=x(1-x)\left(1-e^{-y/\sqrt{\varepsilon}}\right)\left(1-e^{(y-1)/\sqrt{\varepsilon}}\right)\\
&+y(1-y)\left(1-e^{-x/\sqrt{\varepsilon}}\right)\left(1-e^{(x-1)/\sqrt{\varepsilon}}\right),
\end{aligned}
\]
as considered in \cite{Li01}. For this problem, the solution has boundary layers on all sides of the unit square. The test space $V_h=\subset H_0^1(\Omega)$ was chosen to be the space of continuous piecewise linear polynomials with respect to the mesh $\mathcal{T}_h$ and $\M_h=R_hBV_h$ as described in Section \ref{SPLSd4RD}. We used Algorithm \ref{alg:PUCG} to solve \eqref{discrete_problemSPLSRD} with two types of preconditioners: the sBPV preconditioner introduced in Section \ref{sBVP} and a multigrid preconditioner with Gauss-Seidel smoother. We applied the method on both a uniform mesh as well as a Shishkin type mesh, introduced in \cite{Shishkin89}. Details on the implementation of the sBVP preconditioner can be found in the Appendix.

For the Shishkin mesh, we followed the construction outlined in \cite{Roos-Schopf15}. We add it here for completion. We first assume the parameter $N$ is an integer multiple of 8. This refers to the number of mesh intervals in the $x$ and $y$ directions. The mesh itself is the tensor product of two one-dimensional Shishkin meshes $\mathcal{T}_x\times \mathcal{T}_y$. The process for obtaining $\mathcal{T}_x$ (and $\mathcal{T}_y$) is as follows. The interval $[0,1]$ is first decomposed into three subintervals $[0,\lambda]$, $[\lambda, 1-\lambda]$, and $[1-\lambda, 1]$, where 
\begin{equation}\label{lambda_param}
\lambda=\min\left\{\frac{1}{4}, 2\sqrt{\frac{\varepsilon}{c^*}}\ln N\right\} \ \ \mathrm{with} \ \ 0<c^*<c.
\end{equation}
The intervals $[0,\lambda]$ and $[1-\lambda, 1]$ are then partitioned into $N/4$ subintervals of length $\displaystyle{\frac{4\lambda}{N}}$, while the interval $[\lambda, 1-\lambda]$ is partitioned into $N/2$ subintervals of length $\displaystyle{\frac{2(1-2\lambda)}{N}}$. The triangular mesh is obtained by drawing diagonals from the top left to bottom right of each quadrilateral. 

In the case of the Shishkin mesh, we also measured the SPLS error in a balanced norm instead of the norm on ${Q}$. This is due to the fact that for small $\varepsilon$ the $L^2$ part of the norm on $Q$ dominates, leading to an unbalanced norm not adequate to accurately measure the error, see \cite{Lin-Stynes12, Roos-Schopf15}. More precisely, in this case, we compute 
\[
Error:=\left(\|u-u_h\|^2+\varepsilon^{1/2}\|\nabla u - R_h \nabla u_h\|^2\right)^{1/2}.
\]
In the Shishkin mesh case, we used a stopping criteria of $\|q_j\|_Q \leq  10^{-10}$ for $10^{-8}\leq \varepsilon\leq 10^{-4}$, and a stopping criteria of $\|q_j\|_Q\leq 10^{-16}$ for $10^{-14}\leq \varepsilon \leq 10^{-10}$. In the case of a uniform mesh, we used a stopping criteria of $\|q_j\|_Q \leq  10^{-8}$ for $\varepsilon \geq  10^{-6}$.

Table \ref{RD-UniformTab} displays results using uniform meshes, $\M_h=R_h^{\,\text{orth}}AV_h$, and both the sBVP and multigrid preconditioners. For $\varepsilon\leq10^{-3}$ it is shown that the sBVP preconditioner retains a lower iteration count across all levels. The loss of order for small $\varepsilon$ in the uniform mesh case is due to the boundary layers and would need further refinements to resolve them. 

\setlength{\tabcolsep}{0.15em}
\begin{table}[h!] 
\begin{center} 
\begin{tabular}{|*{13}{c|}} 
\hline 
\multicolumn{1}{|c|}{\multirow{2}{*}{\parbox{1.0cm}{\centering Lev/$\varepsilon$}}}& \multicolumn{4}{c|}{$10^{-1}$}& \multicolumn{4}{c|}{$10^{-2}$}& \multicolumn{4}{c|}{$10^{-3}$}\\
\cline{2-13} 
&Error&Order&(a)&(b)&Error&Order&(a)&(b)&Error&Order&(a)&(b)\\
\hline 
1&0.0511&-&6&6&0.0866&-&5&6&0.1490&-&3&5\\ 
\hline 
2&0.0146&1.81&16&21&0.0260&1.74&9&14&0.0680&1.13&5&7\\ 
\hline 
3&0.0041&1.85&30&34&0.0072&1.85&17&25&0.0248&1.46&7&11\\ 
\hline 
4&0.0011&1.88&57&53&0.0019&1.92&31&40&0.0074&1.75&12&20\\ 
\hline 
5&0.0003&1.90&98&81&0.0005&1.95&57&61&0.0020&1.89&22&34\\ 
\hline 
6&7.9e-05&1.92&152&76&0.0001&1.96&99&85&0.0005&1.96&40&53\\ 
\hline 
7&2.1e-05&1.92&215&72&3.2e-05&1.97&150&82&0.0001&1.99&71&76\\ 
\hline 
8&6.1e-06&1.76&220&49&8.4e-06&1.95&190&58&3.2e-05&2.00&108&89\\ 
\hline 
\multicolumn{1}{|c|}{\multirow{2}{*}{\parbox{1.0cm}{\centering Lev/$\varepsilon$}}}& \multicolumn{4}{c|}{$10^{-4}$}& \multicolumn{4}{c|}{$10^{-5}$}& \multicolumn{4}{c|}{$10^{-6}$}\\
\cline{2-13} 
&Error&Order&(a)&(b)&Error&Order&(a)&(b)&Error&Order&(a)&(b)\\
\hline 
1&0.1820&-&3&3&0.1920&-&2&2&0.1960&-&2&2\\ 
\hline 
2&0.1090&0.74&3&5&0.1250&0.62&2&4&0.1290&0.60&2&4\\ 
\hline 
3&0.0605&0.85&4&6&0.0820&0.61&3&5&0.0889&0.54&2&5\\ 
\hline 
4&0.0270&1.16&5&8&0.0505&0.70&3&5&0.0607&0.55&2&5\\ 
\hline 
5&0.0092&1.56&7&13&0.0267&0.92&4&7&0.0395&0.62&3&5\\ 
\hline 
6&0.0026&1.81&14&25&0.0108&1.30&5&9&0.0234&0.76&3&6\\ 
\hline 
7&0.0007&1.92&25&41&0.0034&1.66&8&17&0.0114&1.04&4&8\\ 
\hline 
8&0.0002&1.98&45&62&0.0009&1.86&16&30&0.0042&1.44&5&10\\ 
\hline 
\end{tabular} 
\caption{Uniform, $\M_h=R^{\mathrm{\, orth}}_hBV_h$ (a)It-sBPV (b)It-mgGS}\label{RD-UniformTab}
\end{center} 
\end{table} 

Tables \ref{RD-ShishkinTab} and \ref{RD-ShishkinLumpTab} display results using Shishkin type meshes and the sBVP preconditioner along with both types of trial spaces outlined in Section \ref{SPLSd4RD}. Here, $N$ is related to the level according to Level = $\log_2 (N) - 1$. According to \cite{Lin-Stynes12, Roos-Schopf15}, standard Galerkin methods for \eqref{RDstandard} lead to a covergence rate of $\mathcal{O}(N^{-1}\ln N)$ using piecewise linear approximation. As shown in Tables \ref{RD-ShishkinTab} and \ref{RD-ShishkinLumpTab}, we obtain a convergence rate of $\mathcal{O}\left((N^{-1}\ln N)^2\right)$ using the SPLS method. The numerical tests appear to show that in the case of Shishkin meshes, the sBVP preconditioner appears to be robust with respect to $\varepsilon$ for a fixed stopping criteria. In addition, one notable advantage of the preconditioner outlined in this paper is ease of implementation. All of the required parameters and matrices needed to implement the sBVP preconditioner, given in \eqref{eq:BVPImp} or \eqref{eq:sBVPImp}, are naturally computed in a standard implementation aside from the matrix that relates bases between spaces. 

We mention here that, for the sBPV preconditioner on a Shiskin mesh, we took adavantage of the fact that the Shiskin mesh is topologically equivalent with a uniform mesh. We used a piecewise linear bijection (in each direction) to shift the uniform nodes to the Shiskin nodes. For the implementation of sBPV on Shiskin meshes, we used the same extension and restriction operators as in the case of uniform refinement. We plan to investigate the convergence of the new sBVP preconditioner on nonuniform refinements with suitable extension and restriction operators in a future work. This seems to be a challenging problem by itself and, by the best knowledge of the authors, has not been addressed in the case of Shiskin refinements or more general cases of fitted meshes. 
\begin{remark}
All numerical tests performed with the sBVP preconditioner used the preconditioner given by \eqref{eq:BVPImp} as it results in a lower iteration count. At the expense of a slight increase in iteration count, the form given by \eqref{eq:sBVPImp} can be implemented to reduce the computational time. 
\end{remark}

\setlength{\tabcolsep}{0.25em}
\begin{table}[h!] 
\begin{center} 
\begin{tabular}{|*{10}{c|}} 
\hline 
\multicolumn{1}{|c|}{\multirow{2}{*}{\parbox{1.0cm}{\centering Lev/$\varepsilon$}}}& \multicolumn{3}{c|}{$10^{-4}$}& \multicolumn{3}{c|}{$10^{-6}$}& \multicolumn{3}{c|}{$10^{-8}$}\\
\cline{2-10} 
&Error&Order&It&Error&Order&It&Error&Order&It\\
\hline 
1&0.2030&-&4&0.1970&-&3&0.1970&-&2\\ 
\hline 
2&0.1570&0.89&5&0.1540&0.86&5&0.1540&0.86&4\\ 
\hline 
3&0.1020&1.06&7&0.1020&1.03&6&0.1020&1.03&6\\ 
\hline 
4&0.0539&1.36&11&0.0538&1.35&10&0.0538&1.35&9\\ 
\hline 
5&0.0226&1.70&18&0.0226&1.70&17&0.0226&1.70&15\\ 
\hline 
6&0.0079&1.94&30&0.0079&1.95&30&0.0079&1.95&28\\ 
\hline 
7&0.0025&2.04&50&0.0025&2.04&55&0.0025&2.04&51\\ 
\hline 
8&0.0008&2.05&88&0.0008&2.05&101&0.0008&2.05&95\\ 
\hline
\multicolumn{1}{|c|}{\multirow{2}{*}{\parbox{1.0cm}{\centering Lev/$\varepsilon$}}}& \multicolumn{3}{c|}{$10^{-10}$}& \multicolumn{3}{c|}{$10^{-12}$}& \multicolumn{3}{c|}{$10^{-14}$}\\
\cline{2-10} 
&Error&Order&It&Error&Order&It&Error&Order&It\\
\hline 
1&0.1970&-&3&0.1970&-&3&0.1970&-&2\\ 
\hline 
2&0.1540&0.86&7&0.1540&0.86&7&0.1540&0.86&6\\ 
\hline 
3&0.1020&1.03&11&0.1020&1.03&10&0.1020&1.03&9\\ 
\hline 
4&0.0538&1.35&17&0.0538&1.35&16&0.0538&1.35&15\\ 
\hline 
5&0.0226&1.70&28&0.0226&1.70&27&0.0226&1.70&25\\ 
\hline 
6&0.0079&1.95&50&0.0079&1.95&47&0.0079&1.95&45\\ 
\hline 
7&0.0025&2.04&90&0.0025&2.04&86&0.0025&2.04&83\\ 
\hline 
8&0.0008&2.05&166&0.0008&2.05&159&0.0008&2.05&153\\ 
\hline 
\end{tabular} 
\caption{Shishkin,  $\M_h=R^{\mathrm{\, orth}}_hBV_h$ }\label{RD-ShishkinTab} 
\end{center} 
\end{table}

\setlength{\tabcolsep}{0.25em}
\begin{table}[h!] 
\begin{center} 
\begin{tabular}{|*{10}{c|}} 
\hline 
\multicolumn{1}{|c|}{\multirow{2}{*}{\parbox{1.0cm}{\centering Lev/$\varepsilon$}}}& \multicolumn{3}{c|}{$10^{-4}$}& \multicolumn{3}{c|}{$10^{-6}$}& \multicolumn{3}{c|}{$10^{-8}$}\\
\cline{2-10} 
&Error&Order&It&Error&Order&It&Error&Order&It\\
\hline 
1&0.2200&-&3&0.2140&-&3&0.2130&-&2\\ 
\hline 
2&0.1800&0.70&4&0.1790&0.61&4&0.1790&0.60&4\\ 
\hline 
3&0.1280&0.84&6&0.1280&0.84&6&0.1280&0.84&5\\ 
\hline 
4&0.0776&1.07&8&0.0776&1.06&8&0.0777&1.06&7\\ 
\hline 
5&0.0370&1.45&13&0.0369&1.45&12&0.0370&1.45&11\\ 
\hline 
6&0.0138&1.84&24&0.0137&1.84&21&0.0137&1.84&20\\ 
\hline 
7&0.0044&2.05&48&0.0043&2.06&39&0.0043&2.06&36\\ 
\hline 
8&0.0013&2.08&100&0.0013&2.09&71&0.0013&2.09&67\\ 
\hline 
\multicolumn{1}{|c|}{\multirow{2}{*}{\parbox{1.0cm}{\centering Lev/$\varepsilon$}}}& \multicolumn{3}{c|}{$10^{-10}$}& \multicolumn{3}{c|}{$10^{-12}$}& \multicolumn{3}{c|}{$10^{-14}$}\\
\cline{2-10} 
&Error&Order&It&Error&Order&It&Error&Order&It\\
\hline 
1&0.2130&-&3&0.2130&-&3&0.2130&-&2\\ 
\hline 
2&0.1790&0.60&6&0.1790&0.60&6&0.1790&0.60&6\\ 
\hline 
3&0.1280&0.84&9&0.1280&0.84&8&0.1280&0.84&8\\ 
\hline 
4&0.0777&1.06&13&0.0777&1.06&12&0.0777&1.06&11\\ 
\hline 
5&0.0370&1.45&21&0.0370&1.45&19&0.0370&1.45&18\\ 
\hline 
6&0.0137&1.84&35&0.0137&1.84&34&0.0137&1.84&32\\ 
\hline 
7&0.0043&2.06&64&0.0043&2.06&116&0.0043&2.06&58\\ 
\hline 
8&0.0013&2.09&118&0.0013&2.09&113&0.0013&2.09&108\\ 
\hline 
\end{tabular} 
\caption{Shishkin, $\M_h = R^{\mathrm{\, lump}}_hBV_h$}\label{RD-ShishkinLumpTab} 
\end{center} 
\end{table}


\section{Conclusion}\label{sec:conclusion} 
We presented  a preconditioning technique for the singularly perturbed reaction diffusion problem. We considered the concept of  saddle point reformulation of the problem and the concept of optimal test norm as presented in  \cite{J5onDPG,dem-fuh-heu-tia19}. We showed  the performance of our approach on a  combination of two projection trial spaces and two different  preconditioners that efficently cover a wide range of the parameter $\varepsilon$. The method is also robust with respect to $\varepsilon$. The efficiency of the Uzawa preconditioned CG solver depends on the robustness and efficiency of preconditioners for the discrete {\it optimal norm} on the test space $V_h$.  For quasi-uniform meshes, we introduced a simplified version of the Bramble-Pasciak-Vassilevski preconditioner. The numerical experiements demonstrate the preconditioner performs well even in the case of Shiskin type refinements. Also, in the case of Shishkin meshes, we obtain higher order approximation of the gradient of the solution. 

\section{Appendix: A note on sBVP implementation}\label{sec:appendix}

Using the $(\cdot, \cdot)$ inner product on $V_h$, we can identify $V_h^*$ with $V_h$. The implementation of the sBVP preconditioner defined in \eqref{eq:sBVP} is done by computing the coordinate vector of the action of $B^{sBVP}_J$ on dual vectors. To be more precise, let $M_k$ be the mass matrix for the basis $\{\varphi_i^k,i=\overline{1:n_k}\}$, and let $D_k$ be  the diagonal matrix with entries $D_{ii} = (1,\varphi_i^k)$. We define $E_k$  to be the $n_J \times n_k$  matrix that relates  the bases on $V_k$ and $V_J$. That is, for the bases $\{\varphi_i^k, i=\overline{1:n_k}\}$ of  $V_k$ and  $\{\varphi_i^J,i=\overline{1:n_J}\}$ of $V_J$, we have 
\begin{equation}\label{eq:EkT}
(\varphi_1^k,\varphi_2^k,\dots,\varphi_{n_k}^k)^T = E_k^T (\varphi_1^J,\varphi_2^J,\dots,\varphi_{n_J}^J)^T.
\end{equation}

For $f_h\in V_h$, the (dual) vector in $\R^{n_k}$ is defined by
\[
\underset{\sim_k}{f_h}:=((f_h,\varphi_1^k),(f_h,\varphi_2^k),\dots,(f_h,\varphi_{n_k}^k))^T.
\]
For $w_h  =\sum_{i=1}^{n_k} \alpha_i^k\varphi_i^k \in V_k \subset  V_h$, the coordinate vector in $\R^{n_k}$ is denoted by 
\[
\overset{\sim_k} {w_h}:=(\alpha_1^k,\alpha_2^k,\dots,\alpha_{n_k}^k)^T.
\]
Using \eqref{eq:EkT} it is easy to check that
\begin{equation}\label{eq:DC}
\underset{\sim_k}{f_h} = E_k^T  \, \underset{\sim_J}{f_h}, \ \text{and} \ \  \overset{\sim_J} {w_h} =E_k\, \overset{\sim_k} {w_h},
\end{equation}
and by letting $w_h:= \tilde{Q}_k\,f_h$, we have, due to \eqref{eq:Qtilda} and \eqref{eq:DC}, 
\[
\overset{\sim_J} {w_h} = E_k\, D_k^{-1} E_k^T\,  \underset{\sim_J}{f_h}.
\]

To obtain the contribution of  $\gamma_J I$ in \eqref{eq:sBVP}, we note that $ \overset{\sim_J} {f_h} = M_J^{-1} \underset{\sim_J}{f_h}$. Thus, the matrix version of \eqref{eq:sBVP} is given by
\begin{equation}\label{eq:BVPImp}
B_J^{sBVP} \underset{\sim}{f_J} = \gamma_JM_J^{-1}\underset{\sim}{f_J} + \sum_{j=0}^{J-1}(\gamma_j-\gamma_{j+1})E_jD_j^{-1}E_j^T\underset{\sim}{f_J}.
\end{equation}
To avoid mass matrix inversion, we can  use the equivalence \eqref{eq:MequivD} and further simplify $sBVP$ to the (matrix) version
\begin{equation}\label{eq:sBVPImp}
B_J^{sBVP}\underset{\sim}{f_J} = \gamma_J D_J^{-1}\underset{\sim}{f_J} + \sum_{j=0}^{J-1}(\gamma_j-\gamma_{j+1})E_jD_j^{-1}E_j^T\underset{\sim}{f_J}.
\end{equation}
This way, we avoid mass matrix inversion and the iterative process is faster.
We note that the matrix version of \eqref{eq:BVP}  just uses $M_j$ instead of $D_j$ in \eqref{eq:sBVPImp}.




\def\cprime{$'$} \def\ocirc#1{\ifmmode\setbox0=\hbox{$#1$}\dimen0=\ht0
  \advance\dimen0 by1pt\rlap{\hbox to\wd0{\hss\raise\dimen0
  \hbox{\hskip.2em$\scriptscriptstyle\circ$}\hss}}#1\else {\accent"17 #1}\fi}
  \def\cprime{$'$} \def\ocirc#1{\ifmmode\setbox0=\hbox{$#1$}\dimen0=\ht0
  \advance\dimen0 by1pt\rlap{\hbox to\wd0{\hss\raise\dimen0
  \hbox{\hskip.2em$\scriptscriptstyle\circ$}\hss}}#1\else {\accent"17 #1}\fi}

\end{document}